\documentclass[11pt]{amsart}
\usepackage{amssymb}   
\usepackage{amsmath}
\usepackage{amsthm}

\usepackage{times}
\usepackage{color}
\newtheorem{Theorem}{Theorem}[section]
\newtheorem{Lemma}[Theorem]{Lemma}
\newtheorem{Proposition}[Theorem]{Proposition}
\newtheorem{Corollary}[Theorem]{Corollary}
\newtheorem{Conjecture}[Theorem]{Conjecture}

 \theoremstyle{definition}

\newtheorem{Remark}[Theorem]{Remark}
\newtheorem{Remarks}[Theorem]{Remarks}

\numberwithin{equation}{section}

\newcommand{\Pic}{\operatorname{Pic}}

\newcommand{\Sing}{\operatorname{Sing}}
\newcommand{\Sm}{\operatorname{Sm}}

\renewcommand{\O}{{\mathcal O}}

\newcommand{\p}{{\mathbb P}}

\newcommand{\codim}{\operatorname{codim}}

\newcommand{\ho}{\operatorname{h^0}}

\newcommand{\rk}{\operatorname{rk}}

\newcommand{\defect}{\operatorname{def}}

\def\leq{\leqslant}
\def\geq{\geqslant}
\def\deft{~}

\def\bibaut#1{{\sc #1}}
\def\S{Section\deft}
\def\phi{\varphi}
\def\ro[#1]{{\textcolor{red}{#1}}}
  
\begin{document}

\begin{abstract}An embedded manifold is {\it dual defective} if its dual variety is not a hypersurface. Using the geometry of the variety of lines through a general point, we characterize scrolls among dual defective manifolds. This leads to an optimal bound for the dual defect, which improves results due to Ein. We also discuss our conjecture that every dual defective manifold with cyclic Picard group should be secant defective, of a very special type, namely a local quadratic entry locus variety.
\end{abstract}

\title[On dual defective manifolds]{On dual defective manifolds}
\author[Paltin Ionescu]{Paltin Ionescu*}
\address{\sc Dipartimento di Matematica, Universit\` a degli Studi di Ferrara,\linebreak via Machiavelli, 35, 44121 Ferrara\\Italy\newline \noindent and \newline\indent Institute of Mathematics of the Romanian Academy, P.O. Box 1-764, 014700 Bucharest\\ Romania}
\email{Paltin.Ionescu@unife.it}
\author[Francesco Russo]{Francesco Russo}
\address{\sc Dipartimento di Matematica e Informatica\\
Universit\` a degli Studi di Catania\\
Viale A. Doria, 6\\
95125 Catania\\ Italy}
\email{frusso@dmi.unict.it}

\thanks{*Partially  supported  by the Italian Programme ``Incentivazione alla mobilit\`{a} di studiosi stranieri e italiani residenti all'estero" and by the grant PN-II-ID-PCE-2011-3-0288 (contract no.
132/05.10.2011)}
\maketitle
\section*{Introduction}

The present paper is a natural sequel to our previous work, see \cite{Ru, IR, IR2, IR3} and also \cite{BI}. Using the geometry of the variety of lines passing through the general point of an embedded projective manifold, we characterize scrolls among all dual defective manifolds, and relate the latter to some special secant defective ones. 

We consider $n$-dimensional irreducible non-degenerate complex projective manifolds $X\subset \p^{n+c}$. We call $X$ a {\it prime Fano manifold of index $i(X)$} if its Picard group is generated by the hyperplane section class $H$ and $-K_X=i(X)H$ for some positive integer $i(X)$. One consequence of Mori's work \cite{Mori} is that, for $i(X) \geq \frac{n+1}{2}$, $X$ is {\it covered by lines}, i.e.\ through each point of $X$ there passes a line, contained in $X$. Denote by $a$ the dimension of the variety of lines (belonging to  some irreducible covering family), that pass through a general point of $X$. The ``biregular part'' of Mori Theory (no singularities, no flips, ...), see \cite{Mori2, De, Ko}, provides the natural setting for understanding such manifolds. For instance, as first noticed in \cite{BSW}, when $a\geq \frac{n-1}{2}$, there is a Mori contraction of the covering family of lines. Moreover, its general fiber (which is still covered by lines) has cyclic Picard group, thus being a prime Fano manifold. For prime Fanos, the study of covering families of lines is nothing but the classical aspect in the theory of the {\it variety of minimal rational tangents}, developed by Hwang and Mok in a remarkable series of papers, see e.g.\ \cite{HM, HM2, HM3, Hwang}. 
Recall that $X$ is {\it dual defective} if its dual variety is not a hypersurface. Prime Fanos of high index other than complete intersections are quite rare. For instance, {\it dual defective} manifolds and some special but important {\it secant defective} ones provide such examples. Thus, when ${\rm Pic}(X)=\mathbb Z\langle H\rangle$, the class of {\it local quadratic entry locus} varieties, see \cite{Ru, IR, IR2} gives examples with $i(X)=\frac{n+\delta}{2}$, $\delta$ being the secant defect, while dual defective manifolds have $i(X)=\frac{n+k+2}{2}$, where $k$ is the dual defect, see \cite{Ein}.

Starting from an observation due to Buch \cite{Bu}, we show that when $a\geq n-c$, each line from the given irreducible covering family is part of the contact locus of a suitable hyperplane. We elaborate on this and prove that under some mild condition (which is, conjecturally, automatically fulfilled) for dual defective manifolds with cyclic Picard group each general line is a contact line. This allows us, using also \cite{BFS}, to characterize scrolls among all dual defective manifolds, see Theorem~\ref{scrollimit}.
As a consequence, we get an optimal bound on the dual defect, generalizing one of the main results in Ein's foundational papers \cite{Ein, Ein2}. We also give evidence for our conjecture asserting that dual defective manifolds with cyclic Picard group should be local quadratic entry locus varieties.

\section{Preliminaries}
\begin{enumerate}
\item[$(*)$] {\it Setting, terminology and notation}
\end{enumerate}

Throughout the paper we consider $X\subset \p^N$ an irreducible complex projective manifold
of dimension $n\geq 1$. $X$ is assumed to be non-degenerate and $c$ denotes its codimension, so that $N=n+c$. $X\subset\p^N$ is called {\it quadratic} if it is scheme theoretically an intersection of quadrics. 
For $x\in X$ we let $\mathbf{T}_xX$ denote the (affine) Zariski tangent space to $X$ at $x$, and write $T_xX$ for its projective closure in $\p^N$.
$H$ denotes a hyperplane section (class) of $X$. As usual, $K_X$ stands for the canonical class of $X$. Also, if $Y\subset X$ is a submanifold, we denote by $N_{Y/X}$ its normal bundle. For a vector bundle $E$, $\p(E)$ stands for its projectivized bundle, using Grothendieck's convention.

We let $SX\subset \p^N$ be the {\it secant variety} of $X$, that is the closure of the locus of secant lines. The {\it secant defect} of $X$ is the (nonnegative) number $\delta=\delta(X):=2n+1-\dim(SX)$. We say $X$ is {\it secant defective} when $\delta > 0$.

A secant defective manifold $X\subset \p^N$ is called a {\it local quadratic entry locus} manifold, LQELM for short, if any two general points $x, x' \in X$ belong to a $\delta$-dimensional quadric  $Q_{x,x'}\subset X$. See \cite{Ru, IR} for a systematic study of these special secant defective manifolds.

$X\subset \p^N$ is {\it conic-connected} if any two general points $x, x'\in X$ are contained in some conic $C_{x,x'}\subset X$. Clearly, any LQELM is also conic-connected. Classification results for conic-connected manifolds are obtained in \cite{IR2}, working in the general setting of rationally connected manifolds.

$X\subset \p^N$ is {\it dual defective}, DD for short, if its dual variety $X^*$ is not a hypersurface in the dual projective space $\p^{N*}$.
Its {\it dual defect} is the (positive) number $k:=N-1-\dim(X^*)$.

For a general point $x\in X$, we denote by $\mathcal L_x$ the (possibly empty) scheme of lines contained in $X$ and passing through $x$. We say that $X\subset \p^N$ is {\it covered by lines} if $\mathcal L_x$ is not empty for $x\in X$ a general point. We refer the reader to \cite{De, Ko} for standard useful facts about the deformation theory of rational curves; we shall use them implicitly in the simplest case, that is lines on $X$.

Recall that $X$ is Fano if $-K_X$ is ample. The {\it index} of $X$, denoted by $i(X)$, is the largest integer $j$ such that $-K_X = jA$ for some ample divisor $A$.

\begin{enumerate}
\item[$(**)$] {\it Some standard exact sequences}
\end{enumerate}
Let $V$ be a complex vector space of dimension $N+1$ such that $\p(V)=\p^N$.
Consider the restriction of the Euler sequence on $\p^N$ to $X$
\begin{equation}\label{Euler}
0\to \Omega^1_{\p^N|X}\to V\otimes \O_X(-1)\to \O_X\to 0,
\end{equation}
and the exact sequence on $X$
\begin{equation}\label{tangent}
0\to N_{X/\p^N}^*\to  \Omega^1_{\p^N|X}\to \Omega^1_X\to 0.
\end{equation}
From these exact sequences we deduce
\begin{equation}\label{princ}
0\to N_{X/\p^N}^*(1)\to V\otimes \O_X\to \mathcal P_X\to 0,
\end{equation}
and
\begin{equation}\label{E}
0\to \Omega^1_X(1)\to \mathcal P_X\to \O_X(1)\to 0,
\end{equation}
where $\mathcal P_X$ is the first jet bundle of $\O_X(1)$.

If $\mathbb G(n,N)$ denotes the Grassmannian of $n$-planes in $\p^N$, let $\gamma_X:X\to \mathbb G(n,N)$ be the Gauss map of $X$, associating to a point $x\in X$ the point of $\mathbb G(n,N)$
corresponding to the projective tangent space $T_xX$ to $X$ at $x$. 
By Zak's Theorem on Tangencies, the morphism $\gamma_X$ is
finite and birational, see \cite[I.2.8]{Zak}.

Let $\mathcal U$ be the universal quotient bundle  on $\mathbb G(n,N)$, which is a locally free sheaf of rank $n+1$. We have a natural surjection
$$V\otimes \O_{\mathbb G(n,N)}\to \mathcal U,$$ inducing the  surjection $$V\otimes\O_X\to \gamma_X^*(\mathcal U).$$ 
Then it is easy to see that 
$\mathcal P_X=\gamma_X^*(\mathcal U)$; so, for every closed point $x\in X$ we have $\p(\mathcal P_X\otimes k(x))=T_xX\subset\p(V)$.

The above surjection gives an embedding, over $X$, of $\p(\mathcal P_X)\to X$ into $\p(V)\times X\to X$ in such a way that the restriction $\pi_1$ of the projection $\p^N\times X\to \p^N$
to $\p(\mathcal P_X)$ maps $\p(\mathcal P_X)$ onto $TX=\bigcup_{x\in X}T_xX$. 

By \eqref{E} we get 
$$\gamma_X^*(\det(\mathcal U))=\det(\mathcal P_X)\simeq \omega_X\otimes\O_X(n+1).$$
 
The following simple lemma will provide useful.
\begin{Lemma}\label{norm}
Let $E$ be a spanned vector bundle on $X$ and let $C\subset \p(E):=\p$ be a smooth curve, orthogonal to the tautological line bundle of $\p$.
Let $C'$ be the projection of $C$ to $X$. The following relation between normal bundles holds:
$\deg(N_{C'/X})=\deg(N_{C/\p})+\deg(E|_{C'})$.
\end{Lemma}
\begin{proof}
Observe first that $C$ projects isomorphically onto $C'$, in particular $C'$ is smooth. The result follows by combining the adjunction formula with the well-known expression for the canonical class of $\p$:

$K_{\p}=-\rk(E)L+\pi^*(\det(E)+K_X)$,
where $L$ is the tautological line bundle and $\pi$ the projection from $\p$ to $X$.
\end{proof}

\section{Manifolds covered by lines}

Let $X\subset \p^N$ be as in ($*$). 
Fix some irreducible component, say $\mathcal F$, of the Hilbert scheme of lines on $X\subset \p^N$, such that $X$ is covered by the lines in $\mathcal F$. 
Put $a=: \deg (N_{\ell/X})$, where $[\ell] \in \mathcal F$. 
For  $x\in X$ , let  $\mathcal{F}_x=\{[\ell] \in\mathcal F\mid x\in
\ell\}$. Note that for $x\in X$ general
and for $[\ell] \in\mathcal {F}_x$ we have  $a\geq 0$ and $a=
\dim_{[\ell]}(\mathcal{F} _x)$.
Moreover, we may view $\mathcal{F}_x$ as a closed subscheme of $\p((\mathbf T_xX)^*)\cong \p^{n-1}$.

When the dimension of $\mathcal{F}_x$ is large, the study of manifolds covered by lines is greatly simplified by the following two facts:

First, we may reduce, via a Mori contraction, to the case where the Picard group is cyclic; this is due to Beltrametti--Sommese--Wi\'sniewski, see \cite{BSW}.
Secondly, the variety $\mathcal{F}_x\subset \p^{n-1}$
inherits many of the good properties of $X\subset \p^N$; this is due to Hwang, see \cite{Hwang}. See \cite{BI} for an application of these principles.
\begin{Theorem}\label{reduction}
Assume $a\geq \frac {n-1}2$. Then the following results hold:
\begin{enumerate}
\item {\rm (\cite{BSW})} There is a Mori contraction, say ${\rm cont}_{\mathcal F} : X\to W$, of the lines from $\mathcal F$; let $F$ denote a general fiber of ${\rm cont}_{\mathcal F}$ and let $f$ be its dimension;

\item {\rm (\cite{Wisniewski})} ${\rm Pic}(F) = \mathbb{Z} \langle H_F\rangle$, $i(F)= a+2$ and $F$ is covered by the lines from $\mathcal F$ contained in $F$;

\item {\rm (\cite{Hwang})} $\mathcal {F}_x\subseteq \mathbb{P}^{f-1}$ is smooth irreducible non-degenerate. In particular, $F$ has only one maximal irreducible covering family of lines.
\end{enumerate}
\end{Theorem}

We refer now to the setting in ($* *$). The following considerations are an elaboration of Lemma 1 from \cite{Bu}.

Consider a line $l\subset X$. Then $N_{l/X}\simeq\O_{\p^1}(a_1)\oplus\cdots\oplus\O_{\p^1}(a_{n-1})$ with $a_1\leq a_2\leq\cdots\leq a_{n-1}\leq 1$. Let $a(l)=\#\{i\mid a_i=1\}$.
If $l$ passes through a general point $x\in X$, then $a_1\geq 0$ and $\deg(N_{l/X})=a(l)$.

For every line $l\subset X$ we shall consider the following linear spaces:
\begin{equation}
\big\langle {\textstyle\bigcup\limits_{x\in l}}T_xX\big\rangle,
\end{equation}
which is the linear span of the union of the tangent spaces at points $x\in l$
and 
\begin{equation}
{\textstyle\bigcap\limits_{x\in l}}T_xX.
\end{equation}

By the discussion from ($* *$), we also have the equality
\begin{equation}\label{Tint}
\pi_1(\p((\mathcal P_X)_{|l}))= {\textstyle\bigcup\limits_{x\in l}}T_xX.
\end{equation}

Since we shall be interested in $\bigcup_{x\in l}T_xX$ and in the dimension of its linear span in $\p^N$, we shall analyze  $\mathcal P_{X|l}$,
where for the moment $l\subset X$ is an arbitrary line. From \eqref{princ} we deduce that $\mathcal P_X$  and $\mathcal P_{X|l}$ are generated by global
sections so that
\begin{equation}\label{splittingarbitrary}
\mathcal P_{X|l}\simeq \bigoplus\limits_{j=1}^{n+1-b_0(l)}\O_{\p^1}(b_j(l))\bigoplus \O_{\p^1}^{b_0(l)}
\end{equation}
with $b_j(l)>0$ for every $j=1,\ldots, n+1-b_0(l)$. 

Let $\Pi=\p^b=\p(U)\subset\p^N$. We have a surjection $V\to U$ inducing a surjection $V\otimes\O_X\to U\otimes\O_X$ and hence
an inclusion, over $X$,  $\p(U)\times X\subset\p(V)\times X$. Given a subvariety $Y\subseteq X$ we have that $\Pi\subset T_yX$ for every
$y\in Y$ if and only if the natural surjection $V\otimes\O_Y\to U\otimes \O_Y$ factorizes through $V\otimes \O_Y\to \mathcal P_{X|Y}$, that is if and only if there exists  a surjection $\mathcal P_{X|Y}\to U\otimes \O_Y$.
Thus for a line $l\subset X$ we obtain that if $\Pi=\p^b\subset T_xX$ for every $x\in l$, then $b\leq b_0(l)-1$.

\begin{Proposition}{\rm(cf. \cite[Lemma1]{Bu})}\label{contact} Let notation be as above and let $l\subset X$ be a line passing through a general point of $X$. Then:
\begin{enumerate}
\item 
$$\mathcal P_{X|l}\simeq\O_{\p^1}^{a(l)+2}\oplus\O_{\p^1}(1)^{n-1-a(l)}$$
and 
$$\dim\big(\big\langle {\textstyle\bigcup\limits_{x\in l}}T_xX\big\rangle\big)=N-\ho(N_{X/\p^N}^*(1)_{|l}).$$

In particular the variety $\bigcup_{x\in l}T_xX$ is isomorphic to a linear projection
of a cone with  vertex a linear space of dimension $a(l)+1$ over the Segre variety $\p^1\times \p^{n-2-a(l)}$
so that $$\dim\big(\big\langle {\textstyle\bigcup\limits_{x\in l}}T_xX\big\rangle\big)\leq\min\{N, 2n-1-a(l)\}.$$

\item $$\dim\big({\textstyle \bigcap\limits_{x\in l}}T_xX\big)=a(l)+1$$ and for $x,y\in l$ general we have
$$T_y\widetilde C_x={\textstyle\bigcap\limits_{z\in l}}T_zX,$$ where $\widetilde C_x$ is the irreducible component of the locus
of lines through $x$ to which $l$ belongs.

\item $\ho(N_{X/\p^N}(-1)_{|l})=N-a(l)-1$ and $\ho(N_{X/\p^N}^*(1)_{|l})\geq a(l)+c+1-n$.

\item The following conditions are equivalent:

\begin{itemize}
\item[(i)]
$\ho(N_{X/\p^N}^*(1)_{|l})= a(l)+c+1-n$;
\item[(ii)]
$\dim(\langle \bigcup_{x\in l}T_xX\rangle)= 2n-1-a(l)$;
\item[(iii)]
the splitting-type of the bundle $N_{X/\p^N}(-1)_{|l}$ is $(0,\ldots,0,1,\ldots,1)$.

\end{itemize}
\end{enumerate}
\end{Proposition}
\begin{proof}  From the exact sequence
$$0\to \O_l(-1)\to \mathcal P_{X|l}^*\to T_X(-1)_{|l}\to 0$$
and from $h^0(T_X(-1)_{|l})=a(l)+2$ we deduce $b_0(l)=h^0(\mathcal P_{X|l}^*)=a(l)+2$. Moreover
\begin{align*}n+1-b_0(l)\leq\sum_{j=1}^{n+1-b_0(l)}b_j(l)&=\deg(\mathcal P_{X|l})=(K_X+(n+1)H)\cdot l\\
&=-a(l)-2+n+1=n+1-b_0(l)\end{align*}
so that $b_j(l)=1$ for every $j=1,\ldots,n-1-a(l)$, proving the first assertion.

Let $\Pi=\bigcap_{x\in l}T_xX=\p^b$.  By the previous analysis we deduce $b\leq a(l)+1$. If $l$ passes through the general point $x\in X$,
consider $\widetilde C_x$, the irreducible component of the locus of lines through $x$ to which $l$ belongs, which is a cone
whose vertex contains $x$. We know that $\dim(\widetilde C_x)=a(l)+1$ and that for every $y\in l$ obviously $T_y\widetilde C_x\subset T_yX$.
Since $T_y\widetilde C_x$ does not depend on $y\in l$ for $y\in l$ general, we deduce that, for $x,y\in l$ general,
$T_y\widetilde C_x\subseteq\bigcap_{z\in l}T_zX.$
Therefore $a(l)+1\leq b$ which together with the previous inequality yields $b=a(l)+1$ and
\begin{equation}\label{intformula}
T_y\widetilde C_x={\textstyle \bigcap\limits_{z\in l}}T_zX
\end{equation}
for $x,y\in l$ general points. This proves (2).

The assertions from (3) follow easily from the exact sequences in 
($* *$).
The equivalence between (i) and (ii) from (4) is clear by the second assertion in (1).  Write $N_{X/\p^N}(-1)|_l\simeq \bigoplus\limits_{j=1}^{c}\O_{\p^1}(d_j(l))$, for non negative integers $d_j(l)$. Let $d_0(l)=\#\{i\mid d_i=0\}$ and let $d_1\leq\cdots \leq d_c$. 

We have the exact sequence
\begin{equation}\label{normal}
0\to N_{l/X}(-1)\to N_{l/\p^N}(-1)\simeq \O^{N-1}_l\to N_{X/\p^N}(-1)|_l\to 0,
\end{equation}
yielding
\begin{align*}
n-a(l)-1&=-\deg(N_{l/X}(-1))=\deg(N_{X/\p^N}(-1)|_l)\\&=\sum_{j=d_0(l)+1}^c d_j\geq c-d_0(l). \end{align*}

Moreover, equality holds in the last inequality if and only if condition (iii) from (4) is fulfilled. We deduce that
$h^0(N^*_{X/\p^N}(1)|_l)=d_0(l)\geq a(l)+c+1-n$. This shows the equivalence between (i) and (iii).
\end{proof}
\begin{Corollary}
Keeping the notation and hypotheses from the previous proposition, we have:
\begin{enumerate}

\item
If the splitting-type of the bundle $N_{X/\p^N}(-1)_{|l}$ is $(0,\ldots,0,1,\ldots,1)$, then $a(l)\geq n-1-c$;
\item 
If there is a line $l$ from a covering family such that $N_{X/\p^N}(-1)_{|l}$ is ample, then $a(l)\leq n-1-c$;
\item
If $X$ is quadratic, $n\geq c+2$ and there is a line $l$ from a covering family such that $N_{X/\p^N}(-1)_{|l}$ is ample, then $X$ is a complete intersection.

\end{enumerate}
\end{Corollary}

\begin{proof}
(1) follows from the last point in the proposition. To see (2), recall that by the open nature of ampleness, we may assume that $l$ passes through the general point of $X$. So we may apply item (3) of the proposition, since \hbox{$\ho(N_{X/\p^N}^*(1)_{|l})=0$}. The last point follows from \cite{IR3}, proof of Theorem~3.8 (3), using the previous point.
\end{proof}

\begin{Remark}
If $X$ is quadratic and covered by lines, $N^*_{X/\p^N}(2)$ is spanned, so the splitting-type of the bundle $N_{X/\p^N}(-1)_{|l}$ is $(0,\ldots,0,1,\ldots,1)$, for any line $l$.
\end{Remark}

\section{Dual defective manifolds}

For $X\subset\p^N$, let $X^*\subset\p^{N*}$ be the dual variety of $X$ and let $\defect(X)=N-1-\dim(X^*)$ be the {\it dual defect} of $X$.

 Working in different settings, Mumford in \cite{Mum} and Landman (unpublished) called the attention on this very special but intriguing class of embedded manifolds. They have since then been studied thoroughly by Ein in \cite{Ein, Ein2} and by Beltrametti--Fania--Sommese in \cite{BFS}. See also \cite{LS, Munoz, Munoz2}.

For a hyperplane $H\subset\p^N$  we define the {\it contact locus} of $H$ on $X$ as
\begin{equation}\label{contactlocus}
L=L(H)=\{x\in X\mid T_xX\subseteq H\}=\Sing(X\cap H)\subset X.
\end{equation}
If $[H]\in X^*$  corresponds to a smooth point in $X^*$, then by {\it reflexivity} $L(H)\simeq \p^{\defect(X)}$ is an embedded linear subspace of $\p^N$ contained in $X$. Any line $l\subset L(H)$ with $[H]\in X^*$ 
is called a {\it contact line} on $X$. We recall a basic result on the geometry of contact lines proved differently by Ein in \cite[\S 2]{Ein}.

\begin{Proposition}{\rm (\cite{Ein})} \label{Eindual} 
Let $k=\defect(X)>0$ and let $l\subset L(H)$ be a contact line with $[H]\in X^*$ general. Then
$N_{l/X}\simeq\O_{\p^1}^{\frac{n-k}{2}}\oplus\O_{\p^1}(1)^{\frac{n+k-2}{2}};$
in particular, $n$ and $k$ have the same parity, as first proved by Landman.

\end{Proposition}
\begin{proof}
Apply Lemma \ref{norm} to a line $l\subset L(H)$, where $[H]\in X^*$ is general. Recall that the dual variety is obtained as the image of the natural map (given by the tautological line bundle) from the conormal variety $\p(N_{X/\p}(-1))$ to $\p^{N*}$. Seeing $l$ as contained in the linear $L(H)\simeq \p^k$ which is the fiber of that map, the degree of its normal bundle is $k-1$. Also, $\deg(N_{l/X})=a$ and $\deg(N_{X/\p}(-1)|_{l})=-\deg(N_{l/X}(-1))=n-a-1$, see \ref{normal}.
The result follows.
\end{proof}

Let $k=\defect(X)>0$ and let $x\in X$ be a general point. By Proposition~\ref{Eindual} there exists a unique irreducible component of $\mathcal L_x$ of dimension $\frac{n+k-2}{2}$ containing a given general contact line.
The union of all these irreducible components of $\mathcal L_x$ is called $C_x$. Since $\mathcal L_x$ is smooth and since $\frac{n+k-2}{2}\geq \frac{n-1}{2}$, there exists a unique
irreducible component of $\mathcal L_x$ containing all the general contact lines passing through $x$, i.e. $C_x\subseteq \mathcal L_x$ is an irreducible component of $\mathcal L_x$ of dimension
$\frac{n+k-2}{2}$ and
hence an irreducible smooth subvariety of $\mathcal L_x$. Easy examples like the Segre varieties $\p^1\times\p^{n-1}\subset\p^{2n-1}$ show that in general $C_x\subsetneq \mathcal L_x$.
For every line $[l]\in C_x$, we get $a(l)=\frac{n+k-2}{2}$.

\begin{Proposition}\label{contactline} Let $X\subset\p^N$ be as in {\rm($*$)} with  $\defect(X)=k>0$ and let $x\in X$ be a
general point. If $k\geq n-2c+2$, then for a general  line $[l]\in C_x$
 \[\dim\big(\big\langle {\textstyle \bigcup\limits_{x\in l}}T_xX\rangle)=\frac{3n-k}{2}\quad
\mbox{and}\quad
N_{X/\p^N}(-1)_{|l}\simeq\O_{\p^1}^{\frac{k+2c-n}{2}}\oplus\O_{\p^1}(1)^{\frac{n-k}{2}}.\]
Moreover,  every  line $[l]\in C_x$  is a contact line and for a general $[l]\in C_x$  every hyperplane $H$ such that
$l\subset L(H)$ represents a smooth point of $X^*$. In particular  through a general line $[l]\in C_x\subset\p^{n-1}$ there passes a linearly embedded $\p^{k-1}$
contained in $C_x$, corresponding to lines in $L(H)$ passing through $x$. 

\end{Proposition}
\begin{proof} If $[l]\in C_x$, then $a(l)=\frac{n+k-2}{2}$
so that the hypothesis and  Proposition \ref{contact}
imply 
\begin{equation}\label{dimcont}
\dim\big(\big\langle {\textstyle \bigcup\limits_{x\in l}}T_xX\rangle)\leq 2n-1-a(l)\leq  N-1.
\end{equation}
Thus every  line $[l]\in C_x$  is contained in the contact locus of at least one hyperplane $H$ with $[H]\in (T_xX)^*$.

Let $X^*_x=\{[H]\in X^*\mid T_xX\subseteq H\}=(T_xX)^*\simeq\p^{c-1}$.  
Put
$$Z_x=\{([l],[H])\mid  l\subset L(H)\}\subset C_x\times X^*_x,$$
where $p_x:Z_x\to C_x$ is the restriction of the first projection and $q_x:Z_x\to X^*_x$ the restriction
of the second projection.  For every $[H]\in \Sm(X^*)\cap X^*_x$ we have  $q_x^{-1}([H])=\p^{k-1}$, where $\Sm(X^*)$ denotes the smooth locus of $X^*$. Thus 
$q_x$ is surjective and every irreducible component of $Z_x$ dominating $X^*_x$ has dimension $k+c-2$.
By \eqref{dimcont} also $p_x$ is surjective and by definition of $Z_x$ and $C_x$ every irreducible component
of $Z_x$ dominating $C_x$ dominates $X^*_x$ and vice versa. Let $U=q_x^{-1}( \Sm(X^*)\cap X^*_x)$. Since $p_x$ is proper
and surjective, there exists an open subset $V\subseteq C_x$ such that $p_x^{-1}(V)\subseteq U$. Thus for every $[l]\in V$,
every hyperplane $H$ such that $l\subset L(H)$ is a smooth point of $X^*$ and through $[l]$ there passes a linear embedded
$\p^{k-1}\subset\p^{n-1}$ contained in $C_x$.

Let $[l]\in C_x$ be an arbitrary line and let $F_{[l]}=p_x^{-1}([l])$. By definition 
$F_{[l]}=(\langle \bigcup_{x\in l}T_xX\rangle)^*=\p^{N-b(l)-1}$, where $b(l)=\dim(\langle \bigcup_{x\in l}T_xX\rangle)$.
Thus for a general $[l]\in C_x$
\begin{equation}\label{dimF}
\dim(F_{[l]})=k+c-2-\frac{n+k-2}{2}=\frac{k+2c-2-n}{2}, 
\end{equation}
and
\begin{equation}
n+c-b(l)-1=N-b(l)-1=\dim(F_{[l]})=\frac{k+2c-2-n}{2},
\end{equation}
yielding $b(l)=\dim (\langle \bigcup_{x\in l}T_xX\rangle) =\frac{3n-k}{2}=2n-1-a(l)$.

The decomposition $N_{X/\p^N}(-1)_{|l}\simeq\O_{\p^1}^{\frac{k+2c-n}{2}}\oplus\O_{\p^1}(1)^{\frac{n-k}{2}}$ follows now from Proposition \ref{contact} (4) and the rest is clear.
\end{proof}

\begin{Lemma}\label{3k}
Let $X\subset\p^N$ be as in {\rm($*$)} with  $\defect(X)=k>0$. If $k\geq\frac{n}{3}$,
then $k\geq n-2c+2$.
\end{Lemma}
\begin{proof} By Zak's Theorem on Tangencies, 
$k\leq c-1$. Then
$$n\leq 3k\leq k+2c-2,$$
concluding the proof.
\end{proof}

Let us recall that $X\subset \p^N$ is an {\it $r$-scroll}, if $X\simeq \p(E)$ where $E$ is a rank~$r+1$ vector bundle over some manifold $W$ and the fibers of the projection $\pi:\p(E)\to W$ are linearly embedded in $\p^N$. When $r > \dim(W)$, the scroll $X$ is dual defective and its defect equals $r-\dim(W)$. The next result characterizes scrolls among all dual defective manifolds. 

\begin{Theorem}\label{scrollimit}
Let $X\subset\p^N$ be as in {\rm($*$)} with  $\defect(X)=k>0$. Let $C_x\subset\p^{n-1}$ be as above and put
$\p^m\simeq T \subseteq\p^{n-1}$ to be the linear span of $C_x$ in $\p^{n-1}$. The following conditions are equivalent:
\begin{enumerate}
\item[{\rm (i)}] $X$ is an $\frac{n+k}{2}$-scroll over a manifold of dimension $\frac{n-k}{2}$;
\item [{\rm (ii)}] $\dim(C_x)>2\codim_T(C_x)$ (or equivalently $k>\frac{4m+6-3n}{3}$).
\end{enumerate}
\end{Theorem}
\begin{proof} We only have to prove that (ii) implies (i). Since $n\equiv k$ modulo $2$, it follows that $\dim(C_x)=\frac{n+k-2}{2}\geq[\frac{n}{2}]$. In particular the lines in $C_x$ generate
an extremal ray of $X$ by Theorem~\ref{reduction}~(1). Let $\phi:X\to W$ be the contraction of this ray and let $F$ be a general fiber
of $\phi$. Then $F\subset\p^N$ is a smooth irreducible projective variety such that $\Pic(F)\simeq\mathbb Z\langle H_F\rangle$. Let $f=\dim(F)$ and let $\langle F\rangle$ be the linear span of $F$ in $\p^N$.
Then $C_x\subseteq \p((\mathbf T_xF)^*)=\p^{f-1}$ is smooth irreducible non-degenerate by Theorem~\ref{reduction}~(3).
Thus $m=f-1$. Moreover, by \cite[Theorem (1.2)]{BFS} we have $\defect(F)=k(F)=k+n-f$ so that the hypothesis in (ii) yields $k(F)>\frac{f+2}{3}$.
By Lemma \ref{3k} and Proposition \ref{contactline} every line in $C_x$ is a contact line for $F\subseteq \langle F\rangle$ and $C_x\subseteq\p^{f-1}$
is covered by linear spaces of dimension $k(F)-1>[\frac{\dim(C_x)}{2}]$. From \cite{Sato} (see also \cite{BI} for a simple proof in the spirit of the present paper) it follows that $C_x\subseteq \p((\mathbf T_xF)^*)=\p^{f-1}$ is
a scroll. The condition in (ii) and the Barth--Larsen Theorem, see \cite{BL}, give $\Pic(C_x)\simeq\mathbb Z\langle H_{C_x}\rangle$. So we get $C_x=\p^{f-1}$, hence $F=\p^f$ and $k(F)=f$. Therefore $f=\frac{n+k}{2}>\frac{n}{2}$ and  $X\subset\p^N$ is a scroll
by \cite[Theorem 1.7]{Ein2}.
\end{proof}

\begin{Corollary}\label{extremal} Let $X\subset\p^N$ be as in {\rm($*$)}, with  $\defect(X)=k>0$. Let $\phi:X\to W$ be the contraction whose existence is ensured by Theorem~{\rm\ref{reduction}~(1)}.
Assume that $X$ is not
a scroll. Then $$k\leq \frac{n+2-4\dim(W)}{3}\leq \frac{n+2}{3}.$$
Moreover:
\begin{enumerate}
\item $k=\frac{n+2}{3}$ if and only if\, $N=15$, $n=10$ and $X\subset\p^{15}$
is projectively equivalent to the $10$-dimensional spinorial variety $S^{10}\subset\p^{15}$.
\item If $\dim(W)>0$,  
then  $k=\frac{n+2-4\dim(W)}{3}$ if and only
if\, $\dim(W)=n-10\leq 3$ and $\phi:X\to W$ is a fibration such that the  
general fiber $F\subset\langle F\rangle \subset\p^N$ is isomorphic to  
$S^{10}\subset\p^{15}\subset\p^N$.

\end{enumerate}

\end{Corollary}
\begin{proof} Keeping the notation from the preceding theorem, we have  $k\leq\frac{4m+6-3n}{3}$, $m=f-1$ and $f=n-\dim(W)$. So we may assume that $k=\frac{n+2}{3}$. 
Then $\dim(C_x)=\frac{2(n-1)}{3}$. By Lemma \ref{3k} and
Proposition \ref{contactline} $C_x\subset\p^{n-1}$ is covered by linear spaces of dimension $k-1=\frac{n-1}{3}=\frac{\dim(C_x)}{2}$. Then by \cite{Sato} and \cite{ON}
the variety $C_x\subset\p^{n-1}$ is one of the following:
\begin{enumerate}
\item[{\rm a)}] a scroll;
\item[{\rm b)}] a quadric hypersurface of even dimension;
\item[{\rm c)}] $\mathbb G(1,r)$ Pl\" ucker embedded with $r\geq 4$, or one of its isomorphic projections.
\end{enumerate}

If $n<7$, then $n=4$ and $k=n-2=\dim(C_x)$; reasoning as in the proof of
Theorem \ref{scrollimit}, $X\subset\p^n$ would be a scroll over a curve, contradicting our
assumption. If $C_x\subset\p^{n-1}$ is a (quadric) hypersurface, then $\frac{2(n-1)}{3}=n-2$ yields
$n=4$, obtaining once again a contradiction. Thus we can suppose $n\geq 7$, so that $\Pic(C_x)=\mathbb Z\langle H\rangle$ by the Barth-Larsen Theorem. Therefore
we are necessarily  in case c). Since the secant defect of $\mathbb G(1,r)$ is four, we get:
$n-1\geq 2\dim(C_x)+1-\delta(C_x)=\frac{4n-13}{3}$, yielding $n\leq 10$. On the other hand, we have $\dim(C_x)=2(r-1)\geq 6$, so that $n\geq 10$. Therefore $n=10$, $k=4$ and $i(X)=8$; moreover, the proof of the previous theorem shows that $F=X$, i.e. $\Pic(X)\simeq\mathbb Z\langle H\rangle$. The conclusion of part (1) follows by Mukai's
classification of prime Fano manifolds of index $n-2$, see \cite{Mukai}.

To prove (2), assume that $k=\frac{n+2-4\dim(W)}{3}$. By \cite[Theorem (1.2)]{BFS}
$$k(F)=\frac{n+2-4\dim(W)}{3}+\dim(W)=\frac{f+2}{3},$$
so that by the first part $f=10=n-\dim(W)$ and $F\subset\langle F 
\rangle \subset\p^N$ is isomorphic to $S^{10}\subset\p^{15}$.
\end{proof}

\begin{Remark} The preceding corollary improves one of the main results in \cite{Ein2} stating that if $k>0$ and $X$ is not a scroll, we have $k\leq \frac{n-2}{2}$, with equality only if $X$ is projectively isomorphic either to $\mathbb G(1,4)\subset \p^9$ or to $S^{10}\subset \p^{15}$. If $n\leq 9$ Ein's bound is better than ours; these cases may be recovered by our method, too. 
\end{Remark}

\section{DD versus LQEL manifolds}

We begin by recalling the main known results about LQEL manifolds; complete proofs of the statements in the next theorem may be found in \cite{Ru, IR, IR2}.
\begin{Theorem}\label{LQEL} Assume $X\subset \mathbb{P}^N$ is a LQEL manifold of secant defect $\delta$. Then:

\begin{enumerate}
\item $X$ is a Fano rational manifold with ${\rm rk Pic}(X)\leq 2$.

\item If ${\rm rk Pic}(X)=2$ then $X$ is one of:

\begin{enumerate}
\item $\mathbb{P}^a\times \mathbb{P}^b$ in its Segre embedding, or
\item the hyperplane section of the above, or
\item the blowing-up of $\p^n$ with center a linear space $L$, embedded by the linear system of quadrics through $L$.
\end{enumerate}

\item If ${\rm rk Pic}(X)\!= \!1$, $X\!\cong\! v_2(\mathbb{P}^n)$ or ${\rm Pic}(X)\! = \!\mathbb{Z} \langle H\rangle$ and $i(X) = \frac {n+\delta}{2}$.

\item If $\delta \geq 3$ then $\mathcal{L}_x \subset \mathbb{P}^{n-1}$ is again a LQELM, $S\mathcal{L}_x= \mathbb{P}^{n-1}$, $\dim (\mathcal{L}_x)= \frac{n+\delta}2 -2$, $\delta (\mathcal{L}_x) =\delta -2$. If $\delta \geq \frac n2$, a complete classification is obtained.

\item $X$ is a complete intersection if and only if $X$ is a quadric $\mathbb{Q}^n (\delta=n)$.

\end{enumerate}
\end{Theorem}

\begin{Proposition}
Assume $X\subset \mathbb{P}^N$ is a LQEL manifold of secant defect $\delta$, different from $\mathbb{Q}^n$. Then $\delta \leq c+1$. 
\end{Proposition}
\begin{proof}
Let $a:=\dim(\mathcal{L}_x)$ and recall that $a=\frac{n+\delta -4}2$. We claim that $a\geq n-c$. Indeed, if not, we would have $n-c+1\leq \delta \leq n-2c+2$, so $c=1$, a contradiction. Thus, we may apply \cite{MMT}, Proposition 3.1, giving $a\leq \frac{n+c-3}2$. As $\delta=2a+4-n$, the conclusion follows.
\end{proof} 
\begin{Corollary}{\rm (cf. \cite[Remark 3.3]{Ru})}\label{fu}
The Hartshorne Conjecture on complete intersections, see \cite{Ha}, holds for LQEL manifolds.

\end{Corollary}
\begin{proof}
We have $n-c+1\leq \delta \leq c+1$, unless $X$ is a quadric. It follows that $n\leq 2c$.
\end{proof} 
We also get a new proof of a bound on $\delta$ due to B. Fu, \cite{Fu}.
It expresses the fact that the Hartshorne Conjecture holds for the variety of lines of a LQEL manifold (see \cite{IR3}).
 
\begin{Corollary}{\rm (\cite{Fu})}\label{delta} 
 Assume $X\subset \p^N$ is a LQEL manifold of secant defect $\delta$, different from $\mathbb{Q}^n$. Then $\delta \leq \frac{n+8}3$.
\end{Corollary}
\begin{proof}
We may assume $\delta \geq 3$. From \cite{Ru}, Theorem (2.3), it follows that $\mathcal{L}_x \subset \p^{n-1}$ is also a LQEL manifold. Moreover, it is easy to see that  $\mathcal{L}_x$ is a quadric only if $X$ is. By the preceding Corollary applied to $\mathcal{L}_x\subset \p^{n-1}$, we get $a\leq 2(n-1-a)$, where $\dim \mathcal{L}_x=a=\frac{n+\delta-4}2$. The bound on $\delta$ follows.
\end{proof}
\begin{Conjecture}{\rm(\cite{IR2})} Assume $X$ is a LQELM with ${\rm Pic}(X) \cong \mathbb{Z}\langle H\rangle$. Then $X$ is obtained by linear sections and/or isomorphic projections from a rational homogeneous manifold, in its natural minimal embedding.\end{Conjecture}

We recall that (linearly normal) rational homogeneous manifolds are well understood. In particular, those which are secant defective are known to be LQEL manifolds and are completely classified, see \cite{Ka, Zak}. They turn out to be quadratic manifolds and moreover we have that $\delta \leq 8$ if $X$ is not a quadric. 

\medskip

\begin{Conjecture}\label{DD}{\rm(\cite{IR2})} Any dual defective manifold with cyclic Picard group is a LQELM.
\end{Conjecture}
 
 Theorem~(1.2) from \cite{BFS}, based on Theorem~\ref{reduction}~(1), reduces the study of dual defective manifolds to the case when $\Pic(X)\simeq\mathbb Z\langle H\rangle$. Note that $k\leq 4$ in all known examples other than scrolls.

The Conjecture~\ref{DD} may be cut into two parts, stated as follows:

\begin{enumerate}
\item any dual defective manifold with cyclic Picard group is conic-connected, 
and
\item a dual defective manifold $X\subset \p^N$ as in {\rm($*$)} with cyclic Picard group and dual defect $k$, satisfies:

$k\geq n-c-1$.
\end{enumerate}

Note that if $X$ is both DD and LQELM, with cyclic Picard group, the formula for $\dim(\mathcal{L}_x)$ gives $\delta=k+2$. So, the condition in (2) corresponds to the (obvious) inequality $\delta \geq n-c+1$.
Moreover, from \cite[Proposition~3.2]{IR} it follows that (1) and (2) together imply that $X$ is a LQELM.

\begin{Remarks} 
\begin{enumerate}
\item Condition (2) above implies the inequality $k\geq n-2c+2$, which is the hypothesis in Proposition~\ref{contactline} (use that $n\equiv k$ modulo $2$).
\item By Zak's Theorem on Tangencies, we have $k\leq c-1$, which combined with the inequality in (2) yields $n\leq 2c$ for dual defective manifolds. This would prove that DD manifolds satisfy the Hartshorne Conjecture. Note that all known linearly normal DD manifolds (with cyclic Picard group) are quadratic, so satisfy the Hartshorne Conjecture, as proved in \cite{IR3}.
\item Knowing that (2) holds would lead to a much simpler proof of Corollary~\ref{extremal}, without making use of the elaborate results from \cite{Sato, ON}.

\item If $X$ has cyclic Picard group and is both dual defective and a LQELM, we have seen that $\delta=k+2$; therefore, the upper bounds on $\delta$ in Corollary~\ref{delta}, and on $k$ in Corollary~\ref{extremal}, are the same. They express the condition that $$\dim(\mathcal{L}_x)\leq 2\codim(\mathcal{L}_x, \p^{n-1}),$$ 
that is the Hartshorne condition for the variety of lines.

\end{enumerate}
\end{Remarks}

\end{document}